\documentclass[11pt]{amsart}
\usepackage[colorlinks, citecolor=blue, dvipdfm,
pagebackref]{hyperref}

\usepackage{graphicx}
\usepackage[all]{xy}
\usepackage{extarrows}
\usepackage{tikz}
\usepackage{extarrows}
\usepackage{xypic}
\usepackage{exscale}
\usepackage{relsize}

\newtheorem{theorem}{Theorem}[section]
\newtheorem{lemma}[theorem]{Lemma}

\theoremstyle{definition}
\newtheorem{assertion}[theorem]{Assertion}
\newtheorem{definition}[theorem]{Definition}
\newtheorem{question}[theorem]{Question}
\newtheorem{example}[theorem]{Example}

\newtheorem{proposition}[theorem]{Proposition}
\newtheorem{corollary}[theorem]{Corollary}
\newtheorem{remark}[theorem]{Remark}

\theoremstyle{remark}

\newcommand{\be}{\begin{equation}}
\newcommand{\ee}{\end{equation}}

\newcommand{\Z}{\mathbb{Z}}

\numberwithin{equation}{section}

\begin{document}
\title[Topological obstructions and Calabi-Yau manifolds]{Topological obstructions to geometric positivity and negativity on Calabi-Yau manifolds}
\author{Ping Li}
\address{School of Mathematical Sciences, Fudan University, Shanghai 200433, China}

\email{pinglimath@fudan.edu.cn\\
pinglimath@gmail.com}

\thanks{The author was partially supported by the National
Natural Science Foundation of China (Grant No. 12371066).}

%    General info

%\date{January 1, 2001 and, in revised form, June 22, 2001.}

%\dedicatory{Dedicated to Professor Martin Guest on the occasion of his 70th birthday.}

\begin{abstract}
We study whether the topology underlying a Calabi-Yau manifold can support natural geometric positivity or negativity structures. In even complex dimension $n \geq 4$ (assuming $b_2=1$ when $n \geq 6$), we strengthen a theorem of Oguiso-Peternell by proving that a Calabi-Yau manifold is not homeomorphic to a weak Fano $n$-fold. The same obstruction applies to K\"{a}hler manifolds with quasi-positive holomorphic sectional curvature. A transformation-group analogue excludes, in particular, symplectic manifolds admitting Hamiltonian circle actions with isolated fixed points. On the negative side, we show that the fundamental group of a Calabi-Yau manifold is not isomorphic to that of a K\"{a}hler hyperbolic manifold. Taken together, these results exhibit a common topological rigidity separating Calabi-Yau manifolds from several fundamental classes governed by geometric positivity or negativity.
\end{abstract}

\maketitle

\section{Introduction}\label{section-introduction}
Throughout this paper, all smooth and complex manifolds are assumed to be closed,
connected, and oriented. Every complex manifold is equipped with the canonical
orientation induced by its complex structure. Unless otherwise stated, all complex
manifolds have complex dimension $n$.

By a \emph{Calabi-Yau} manifold $M_0$, we mean a K\"{a}hler manifold whose first Chern class satisfies $$c_1(M_0)=0\in H^{1,1}(M_0;\mathbb{R}).$$
A projective manifold $M_+$ (resp. $M_-$) is called \emph{Fano} (resp. \emph{canonically polarized}) if $c_1(M_+)>0$ (resp. $c_1(M_-)<0$); equivalently, the anti-canonical bundle $-K_{M_+}$ (resp. canonical bundle $K_{M_-}$) is ample. By the Calabi-Yau theorem (\cite{Yau}), the Calabi-Yau, Fano and canonically polarized conditions are characterized by the existence of K\"{a}hler metrics with zero, positive and negative Ricci curvature respectively. In the canonically polarized case, the metric can moreover  be chosen to be K\"{a}hler-Einstein by the Aubin-Yau theorem (\cite{Aub,Yau}). Throughout the paper, the symbols $M_0$, $M_+$ and $M_-$ are \emph{reserved} for manifolds of these three types.

When $n=1$, the underlying homeomorphism types of $M_0$, $M_+$ and $M_-$ are pairwise distinct. When $n=2$, neither $M_+$ nor $M_-$ can be homeomorphic to $M_0$ (see Prop. \ref{positivesurface} and \ref{negativesurface}). By contrast, the smooth four-manifold underlying the blow-up of $\mathbb{C}P^2$ at eight points in general position, $$\mathbb{C}P^2\sharp8\overline{\mathbb{C}P^2},$$
admits both a del Pezzo Fano structure, denoted by $S_+$, and a canonically polarized structure, denoted by $S_-$. The former is classical, whereas the latter was observed by Catanese and LeBrun in \cite{CL} (see also \cite{Ko}). Consequently, for every positive integer $k$, the \emph{even}-dimensional Fano manifold $(S_+)^k$ is diffeomorphic to the canonically polarized manifold $(S_-)^k$. Moreover, these manifolds admit K\"{a}hler-Einstein metrics whose Ricci curvatures have opposite signs (\cite[Thm 1]{CL}). The positive case is due to Tian-Yau's work (\cite{TY87}), while the negative case follows from the Aubin-Yau theorem. More complex even-dimensional examples can be found in \cite{RS}. To the best of the author's knowledge, no analogous examples are known in \emph{odd} complex dimensions.

This naturally raises the following question in dimensions $n\geq3$: can a Calabi-Yau manifold be homeomorphic to a Fano or canonically polarized manifold?
Even in complex dimension three, it remains unknown whether a smooth six-dimensional manifold can support both Fano and Calabi-Yau complex structures. For example, Nakamura (\cite{Na}) and Campana-Peternell (\cite{CP}) asked  whether there exists a Calabi-Yau threefold that is orientation-preservingly homeomorphic to either a smooth cubic threefold in $\mathbb{C}P^4$, or a smooth complete intersection of two quadrics in $\mathbb{C}P^5$. This question remains open and was highlighted again in \cite[p.138]{OP}.

In view of the absence of existence results, the Calabi-Yau case appears to be more rigid. Indeed strong obstructions show that, in many cases, a Calabi-Yau manifold cannot be orientation-preservingly homeomorphic to a Fano manifold. For instance, Oguiso-Peternell proved the following result (\cite[Thm 1.4]{OP}).
\begin{theorem}[Oguiso-Peternell]\label{thm-Oguiso-Peternell}
Let $M_0$ be a Calabi-Yau manifold of even complex dimension $n\geq4$. If $n\geq6$, assume in addition that the second Betti number $b_2(M_0)=1$. Then
$M_0$ is not orientation-preservingly homeomorphic to a Fano manifold.
\end{theorem}
Thus, in many even complex dimensions, and especially when $b_2=1$, the topology of a Calabi-Yau manifold distinguishes it from that of a Fano manifold.
For the comparison between Calabi-Yau and canonically polarized manifolds, however, neither existence nor nonexistence results seem previously to have been known.

As noted above, Fano and canonically polarized manifolds can be characterized by the existence of K\"{a}hler metrics with positive and negative Ricci curvature respectively. This suggests broadening the preceding homeomorphism problem beyond these two classes and asking whether the underlying topological type of a Calabi-Yau manifold can support other natural geometric structures exhibiting positivity or negativity. This leads to the central question of the paper.
\begin{question}\label{Question}
Can the homeomorphism or diffeomorphism type underlying a Calabi-Yau manifold support geometric structures satisfying natural positivity or negativity conditions?
\end{question}

The purpose of this paper is to give several negative answers to Question \ref{Question}. We show that, in several important settings, the topology underlying a Calabi-Yau manifold cannot support such structures. Theorem \ref{first thm} strengthens Theorem \ref{thm-Oguiso-Peternell} in two directions: the Fano condition is weakened to weak Fano, and the homeomorphism is not required to preserve orientation. Theorem \ref{second thm} gives a parallel  obstruction for K\"{a}hler metrics with quasi-positive holomorphic sectional curvature. Theorem \ref{third refined thm} provides a transformation-group analogue; in particular, Corollary \ref{third thm} excludes symplectic manifolds admitting Hamiltonian circle actions with isolated fixed points. Theorem \ref{fifth thm} shows that the fundamental group of a Calabi-Yau manifold cannot be isomorphic to that of a K\"{a}hler hyperbolic manifold.

\subsection*{Organization of the paper}
Section \ref{section-main results} states the main results and provides the relevant background. Section \ref{section-preliminary} reviews preliminaries on the $\hat{A}$-genus, the Todd genus, and the Bogomolov-Beauville decomposition. Sections \ref{proof of 1thm}-\ref{proof of 5thm} contain the proofs of the main theorems.

Since scalar curvature is the weakest of the classical curvature conditions considered here, Section \ref{scalar curvature} discusses which signs of scalar curvature can occur for Riemannian metrics on Calabi-Yau manifolds. Finally, Section \ref{appendix} is an appendix proving that neither a Fano surface nor a canonically polarized surface can be homotopy equivalent to a Calabi-Yau surface. The arguments in Sections \ref{scalar curvature} and \ref{appendix} are standard and are included for completeness and the reader's convenience.

\section{Main results}\label{section-main results}
Recall that a K\"{a}hler manifold $M$ is called \emph{weak Fano} if $c_1(M)$ is nef and big. Our first result strengthens \cite[Thm 1.4]{OP} in two directions: it replaces the Fano condition by the weaker weak Fano condition and removes the requirement that the homeomorphism preserve orientation.
\begin{theorem}\label{first thm}
Let $M_0$ be a Calabi-Yau manifold of even dimension $n\geq4$. If $n\geq6$, assume in addition that $b_2(M_0)=1$. Then $M_0$ is not homeomorphic to a weak Fano manifold.
\end{theorem}
Our proof of Theorem \ref{first thm} follows the general strategy of \cite{OP}, but requires additional
ingredients from both algebraic geometry and topology. By the Demailly-Paun bigness criterion (\cite{DP}), a K\"{a}hler manifold with quasi-positive first Chern class has $c_1(M)$ nef and big. We therefore obtain the following consequence.
\begin{corollary}\label{coro}
Let $M_0$ be a Calabi-Yau manifold of even dimension $n\geq4$. If $n\geq6$, assume in addition that $b_2(M_0)=1$. Then $M_0$ cannot be homeomorphic to a K\"{a}hler manifold whose first Chern class is quasi-positive.
\end{corollary}

In particular, Theorem \ref{first thm} and Corollary \ref{coro} imply that the underlying topological type of such a Calabi-Yau manifold cannot carry a K\"{a}hler structure with quasi-positive Ricci curvature.

For a K\"{a}hler metric $g$, the relationship between its Ricci curvature $\text{Ric}(g)$ and its holomorphic sectional curvature $\text{HSC}(g)$ is subtle (see \cite[p.181]{Zhe}). The sign of either $\text{Ric}(g)$ or $\text{HSC}(g)$ determines that of scalar curvature, while both of them are determined by the sign of holomorphic bisectional curvature (see \cite[\S7.5]{Zhe} for details). In general, however, the sign of $\text{Ric}(g)$ and that of $\text{HSC}(g)$ do not imply one another. In this sense, the two positivity conditions may be viewed as occupying comparable but generally incomparable positions in the hierarchy of K\"{a}hler curvature conditions.

Motivated by this parallel, our second main result gives a
holomorphic-sectional-curvature counterpart of Theorem \ref{first thm}.
\begin{theorem}\label{second thm}
Let $M_0$ be a Calabi-Yau manifold of even dimension $n\geq4$. If $n\geq6$, assume in addition
that $b_2(M_0) = 1$. Then $M_0$ is not homeomorphic to a K\"{a}hler manifold with quasi-positive holomorphic sectional curvature.
\end{theorem}

Topological and differential restrictions imposed by curvature positivity often have
striking analogues in transformation group theory. A prominent example is provided by
the $\hat{A}$-genus: for a spin manifold, its vanishing may follow either from the existence of
a Riemannian metric with quasi-positive scalar curvature (\cite{Lic}), or from a (nontrivial) smooth
compact connected Lie group action (\cite{AH}). We refer to \cite[Section 1]{Li} for further discussion of
this analogy.

Our next result may therefore be viewed as a transformation-group analogue of Theorems \ref{first thm} and \ref{second thm}.
\begin{theorem}\label{third refined thm}
Let $M$ be a simply connected smooth manifold such that all odd-dimensional Betti numbers vanish. Suppose that $M$ admits a (nontrivial) smooth compact connected Lie group action. Then $M$ is not homeomorphic to a Calabi-Yau manifold.
\end{theorem}
\begin{remark}\label{G/P}
Classical such examples are rational homogeneous manifolds $G/P$, where $G$ is a semisimple complex Lie group and $P$ a parabolic subgroup. Such manifolds are indeed Fano and therefore simply-connected.
\end{remark}

It turns out that symplectic manifolds admitting Hamiltonian circle actions with isolated fixed points satisfy the assumptions in Theorem \ref{third refined thm} (see Section \ref{proof of 3thm-1}). Hence we have
\begin{corollary}\label{third thm}
No Calabi-Yau manifold is homeomorphic to a symplectic manifold admitting a Hamiltonian circle action with isolated fixed points.
\end{corollary}
\begin{remark}
Rational homogeneous manifolds $G/P$ also satisfy the assumptions in Corollary \ref{third thm}, as noted in \cite{Fr}. Examples of non-K\"{a}hler symplectic manifolds admitting Hamiltonian circle actions with isolated fixed points can be found in \cite[p. 2]{LP}.
\end{remark}

Two fundamental notions around hyperbolicity in complex geometry are \emph{Kobayashi
hyperbolicity} and \emph{K\"{a}hler hyperbolicity}, introduced respectively by Kobayashi (\cite{Kob2}) and Gromov (\cite{Gr}). A compact complex manifold is Kobayashi hyperbolic if it does not contain a nonconstant entire curve, whereas K\"{a}hler hyperbolicity is more involved and is defined in Section \ref{kahlerbyperbolic}.

Gromov observed that K\"{a}hler hyperbolicity implies Kobayashi
hyperbolicity (see \cite[Cor.4.2]{CY}). A weak form of the Kobayashi conjecture predicts that
Calabi-Yau manifolds are \emph{not} Kobayashi hyperbolic (\cite[Prob. C.1]{Kob2}, \cite{KLV}, \cite{KV}); this remains
widely open. Our final result may be regarded as partial evidence in this direction.
\begin{theorem}\label{fifth thm}
Let $M_0$ be a Calabi-Yau manifold and $(M,g)$ a K\"{a}hler hyperbolic manifold. Then
$\pi_1(M_0)$ is not isomorphic to $\pi_1(M)$.
Consequently, $M_0$ and $M$ are not homotopy equivalent.
\end{theorem}

\section{Preliminaries}\label{section-preliminary}
\subsection{The $\hat{A}$-genus}
The \emph{$\hat{A}$-genus} of manifolds is a typical example of the oriented genus introduced by Hirzebruch (\cite{Hi}) and defined as follows. Let the total Pontrjagin class $p(X)$ of a manifold $X$ be formally decomposed as
\be\label{Pondecom} p(X)=1+p_1(X)+p_2(X)+\cdots=:\prod_{i\geq 1}(1+y_i)\in H^{4\ast}(X;\mathbb{Z}),\ee
i.e., the $j$-th Pontrjagin class $p_j(X)$ is represented as the $j$-th elementary symmetric function of the variables $y_1,y_2,\ldots$. Then
\be\label{hatdef}\hat{A}(X):=<\prod_{i\geq1}
\frac{\sqrt{y_i}/2}{\sinh(\sqrt{y_i}/2)},~[X]>\in\mathbb{Q},\ee
where $[X]$ is the fundamental class determined by the orientation. By definition $\hat{A}(X)=0$ unless the dimension of $X$ is divisible by four.

It is well-known that $\hat{A}(X)\in\mathbb{Z}$ whenever $X$ is \emph{spin}, which was observed by Borel and Hirzebruch (\cite[\S 25]{BH}). This fact both motivated and was later explained by the Atiyah-Singer index theorem, which showed that $\hat{A}(\cdot)$ of a spin manifold is the index of its Dirac operator (\cite[\S 5]{AS}).

The following classical vanishing results are due to Lichnerowicz (\cite{Lic}) and Atiyah-Hirzebruch (\cite{AH}) respectively.
\begin{theorem}[Lichnerowicz, Atiyah-Hirzebruch]\label{Lich-AH}
Let $X$ be a spin manifold. Then $\hat{A}(X)=0$ if $X$ admits either a quasi-positive scalar curvature Riemannian metric, or a nontrivial compact connected Lie group action.
\end{theorem}

The next fact is a consequence of the topological invariance of rational Pontrjagin classes due to Novikov (\cite{No}).
\begin{lemma}\label{hattopoinv}
Let $X$ and $Y$ be two homeomorphic manifolds. Then $\hat{A}(X)=0$ if and only if $\hat{A}(Y)=0$.
\end{lemma}
\begin{proof}
Let $$f:~X\overset{\cong}{\longrightarrow}Y$$ be a homeomorphism, and denote by $$\hat{\mathfrak{A}}(X):=\prod_{i\geq1}
\frac{\sqrt{y_i}/2}{\sinh(\sqrt{y_i}/2)},$$
in the notation of (\ref{Pondecom}) and (\ref{hatdef}). The homeomorphic invariance of rational Pontrjagin classes, as well as the fact that torsion elements contribute nothing to integration, yield
\be
\begin{split}
\hat{A}(X)&=<\hat{\mathfrak{A}}(X),~[X]>\\
&=<f^{\ast}(\hat{\mathfrak{A}}(Y)),~[X]>\\
&=<\hat{\mathfrak{A}}(Y),~f_{\ast}([X])>\\
&=<\hat{\mathfrak{A}}(Y),\pm[Y]>\\
&=\pm\hat{A}(Y).
\end{split}\nonumber\ee
Here ``$\pm$" depends on whether or not $f$ is orientation-preserving.
\end{proof}

\subsection{The Todd genus}
Assume now that $(X,J)$ is an almost-complex manifold and its total Chern class $c(X):=c(X,J)$ is given by
\be\label{Cherndecom}c(X)=1+c_1(X)+c_2(X)+\cdots=:\prod_{i\geq1}(1+x_i)\in H^{2\ast}(X;\mathbb{Z}).\ee
Such $x_i$ are usually called \emph{formal Chern roots} of $X$. Then the \emph{Todd genus} of $X$, $\text{Td}(X)$, is a typical example of the complex genus defined by
\be\text{Td}(X):=
<\prod_{i\geq1}\frac{x_i}{1-e^{-x_i}},~[X]>\in\mathbb{Z}.\ee The integrality of $\text{Td}(X)$ is due to the (generalized) Hirzebruch-Riemann-Roch theorem (\cite{Hi},\cite{AS}).

When the almost-complex structure $J$ is integrable, the H-R-R implies that
\be\label{Todd genus}\text{Td}(X)=\chi(X,\mathcal{O}_X):=\sum_{q\geq0}(-1)^qh ^{0,q},\ee
where $h^{0,q}$ are the Hodge numbers and $\chi(X,\mathcal{O}_X)$ is usually called the \emph{holomorphic Euler characteristic}.

If the total Chern class of $(X,J)$ has the decomposition (\ref{Cherndecom}), the total Pontrjagin class then has the following form
$$p(X)=\prod_{i\geq1}(1+x_i^2),$$
and hence we have
\be\hat{A}(X)=<\prod_{i\geq1}
\frac{x_i/2}{\sinh(x_i/2)},~[X]>.\ee

The following identity will play a key role in relating the $\hat{A}$-genus and the Todd genus.
\be\label{relation}\frac{x/2}{\sinh(x/2)}=\frac{x}{1-e^{-x}}\cdot e^{-x/2}.\ee
In particular, we have
\begin{lemma}\label{equal lemma}
For a Calabi-Yau manifold the $\hat{A}$-genus and the Todd genus coincide.
\end{lemma}
\begin{proof}
Let $x_i$ be the formal Chern roots of a Calabi-Yau manifold $M$. Then
\be\begin{split}
\hat{A}(M)&=<\prod_i\frac{x_i/2}{\sinh(x_i/2)},~[M]>\\
&=<\prod_i\big(\frac{x_i}{1-e^{-x_i}}\cdot e^{-x_i/2}\big),~[M]>\quad\big(\text{by (\ref{relation})}\big)\\
&=<(\prod_i\frac{x_i}{1-e^{-x_i}})e^{-c_1(M)/2},~[M]>\\
&=<\prod_i\frac{x_i}{1-e^{-x_i}},~[M]>\\
&=\text{Td}(M).
\end{split}
\ee
\end{proof}

\subsection{The Bogomolov-Beauville decomposition}
\begin{definition}\label{def}
\begin{enumerate}
\item
An $n$-dimensional simply connected Calabi-Yau manifold is called \emph{strict Calabi-Yau} (\emph{SCY} for short) if $n\geq3$, $h^{n,0}=1$ and $h^{p,0}=0$ for $1<p<n$.

\item
A $2n$-dimensional simply connected K\"{a}hler manifold is called \emph{hyperK\"{a}hler} if there exists an everywhere nondegenerate holomorphic two-form $\sigma$ such that
$$H^{2p,0}(M)=\mathbb{C}\cdot\sigma^p,\quad H^{2p+1,0}(M)=0,\quad 1\leq p\leq n.$$
\end{enumerate}
\end{definition}
\begin{remark}
The odd Chern classes of a hyperK\"{a}hler manifold vanish in real cohomology. In particular, every hyperK\"{a}hler manifold is Calabi-Yau in the sense used here.
\end{remark}
The Todd genera of SCY and hyperK\"{a}hler manifolds follow immediately from (\ref{Todd genus}), and we record the values for later use.
\begin{lemma}\label{Todd genus of SCY and hyperKahler}
The Todd genus of an $n$-dimensional SCY manifold is $1+(-1)^n$, and that of a $2n$-dimensional hyperK\"{a}hler manifold is $n+1$.
\end{lemma}

We shall use the following form of the Bogomolov-Beauville decomposition theorem (see \cite{Be}).
\begin{theorem}[Bogomolov-Beauville decomposition]\label{BB decomposition}
Every Calabi-Yau manifold admits a finite \'{e}tale cover of the form
$$T\times\prod_iY_i\times\prod_jZ_j,$$
where $T$ is a complex torus, each $Y_i$ is a SCY manifold, and each $Z_j$ is hyperK\"{a}hler.
\end{theorem}

\section{Proof of Theorem \ref{first thm}}\label{proof of 1thm}
\subsection{Calabi-Yau manifolds with vanishing $\hat{A}$-genus}
The proof is to derive a contradiction by assuming that $M_0$ is homeomorphic to $M$ under the conditions in Theorem \ref{first thm}, as done in \cite[Thm 1.4]{OP}.

The following lemma, on which both Theorems \ref{first thm} and \ref{second thm} are based, is essentially contained in \cite[Thm 3.1]{OP}.
\begin{lemma}\label{BB lemma}
A simply-connected $M_0$ with $\hat{A}(M_0)=0$ is of the form
$$M_0=M_1\times M_2,$$
where $M_1$ is an odd-dimensional SCY manifold, and $M_2$ is either a point or a product of some SCY and hyperK\"{a}hler manifolds.
\end{lemma}
\begin{proof}
By Theorem \ref{BB decomposition}, $M_0$ can be written as a product of some SCY and hyperK\"{a}hler manifolds. Lemmas \ref{equal lemma} and \ref{Todd genus of SCY and hyperKahler} imply that, among SCY and hyperK\"{a}hler manifolds, exactly \emph{odd}-dimensional SCY manifolds have vanishing $\hat{A}$-genera.  Since $\hat{A}(\cdot)$ is multiplicative with respect to products, the decomposition contains at least one odd-dimensional SCY factor.
\end{proof}

\subsection{Proof of Theorem \ref{first thm}}\label{proof of 1thm-2}
Suppose, for contradiction, that $M_0$ in Theorem \ref{first thm} is homeomorphic to a weak Fano manifold $M$. The proof will follow from several assertions.

\begin{assertion}\label{first assertion}
The K\"{a}hler manifold $M$ is simply connected and projective.
\end{assertion}
\begin{proof}
Since $-K_M$ is big, $M$ is Moishezon (\cite[p.88]{MM}). By Moishezon's theorem (\cite[p.95]{MM}), a K\"{a}hler  Moishezon manifold is projective. A projective manifold with nef and big anti-canonical bundle is rationally connected, hence simply connected (\cite[Thm 1]{Zha}).
\end{proof}

\begin{assertion}\label{second assertion}
The anti-canonical bundle $-K_M$ is divisible by two in $\text{Pic}(M)$, i.e., there exists some line bundle $L\in\text{Pic}(M)$ such that $-K_M=2L$.
\end{assertion}
\begin{proof}
By Assertion \ref{first assertion}, both $M$ and $M_0$ are simply connected. The universal coefficient theorem therefore implies that $H^{1,1}(M_0;\Z)$ is torsion-free. Since $c_1(M_0)$ vanishes in real cohomology, it follows that $$c_1(M_0)=0\in H^{1,1}(M_0;\Z).$$
Thus the second Stiefel-Whitney class
$$w_2(M_0)=0\in H^2(M_0;\Z_2)$$
as it is the \text{mod-$2$} reduction of $c_1(M_0)$. By the Wu formula, Stiefel-Whitney classes are homotopy type invariants (\cite[p.130]{MS}) and consequently $$w_2(M)=0\in H^2(M;\Z_2).$$
Since $w_2(M)$ is the \text{mod-$2$} reduction of $c_1(M)$, this means that $c_1(M)$ is divisible by two in $H^{1,1}(M;\Z)$. Simple connectedness of $M$ then implies that the map
$$c_1(\cdot):~\text{Pic}(M)\longrightarrow H^{1,1}(M;\Z)$$
is indeed an isomorphism. Thus $-K_M$ is also divisible by two in $\text{Pic}(M)$.
\end{proof}

\begin{assertion}\label{third assertion}
We have $$\hat{A}(M)=\chi\big(M,\mathcal{O}_M(-L)\big)=0,$$
where $L$ is such that $-K_M=2L$ as in Assertion \ref{second assertion}.
\end{assertion}
\begin{proof}
Let $x_i$ be formal Chern roots of $M$. The Hirzebruch-Riemann-Roch theorem tells us that
\be\begin{split}
\chi\big(M,\mathcal{O}_M(-L)\big)&=<\big(\prod_{i}
\frac{x_i}{1-e^{-x_i}}\big)\cdot e^{-c_1(L)},~[M]>\\
&=<\big(\prod_{i}
\frac{x_i}{1-e^{-x_i}}\big)\cdot e^{-\sum_ix_i/2},~[M]>\quad(\text{by $-K_M=2L$})\\
&=<\prod_{i}\frac{x_i/2}{\sinh(x_i/2)},~[M]>\quad\big(\text{by (\ref{relation})}\big)\\
&=\hat{A}(M).
\end{split}\ee

It remains to show that $\chi\big(M,\mathcal{O}_M(-L)\big)=0$. Since $-K_M=2L$ is nef and big, so is $L$. The Kawamata-Viehweg vanishing theorem (\cite[p.74]{Kob}) therefore gives
$$H^p\big(M,\mathcal{O}_M(-L)\big)=
H^p\big(M,\mathcal{O}_M(K_M+L)\big)=0,\quad p\geq 1.$$

Moreover $H^0\big(M,\mathcal{O}_M(-L)\big)=0$. Indeed, if there were an effective divisor $D\in|-L|$, then
$$0\leq\int_DL^{n-1}=\int_ML^{n-1}\cdot(-L)=-\int_ML^n,$$
which contradicts the bigness of $L$.
\end{proof}

\begin{assertion}\label{fourth assertion}
Theorem \ref{first thm} holds true.
\end{assertion}
\begin{proof}
By Assertion \ref{third assertion} and Lemma \ref{hattopoinv} we have $\hat{A}(M_0)=0$. The case of $n=4$ is incompatible with Lemma \ref{BB lemma} and hence $M_0$ is not homeomorphic to $M$. If the even dimension $n\geq6$, Lemma \ref{BB lemma} then implies that $M_0=M_1\times M_2$ with $\dim M_2>0$. This contradicts $b_2(M_0)=1$. In either case we obtain a contradiction. This proves Theorem \ref{first thm}.
\end{proof}

\section{Proof of Theorem \ref{second thm}}\label{proof of 2thm}
\subsection{Holomorphic sectional curvature}\label{subsection HSC}
Let $(M,g)$ be a K\"{a}hler manifold, and $\nabla$ the associated  Levi-Civita connection of the K\"{a}hler metric $g$. For any $(1,0)$-type tangent vector fields $X$, $Y$, $Z$, $W$, the (complexified) curvature tensor $R$ of $g$ is defined by
$$R(X,\overline{Y},Z,\overline{W}):=
g(\nabla_X\nabla_{\overline{Y}}Z-
\nabla_{\overline{Y}}\nabla_XZ-
\nabla_{[X,\overline{Y}]}Z,\overline{W}).$$
By the tensor property of $R$ it turns out that $R$ can indeed be defined pointwise.

Under a local coordinate system $\{z^1,\ldots,z^n\}$, the components of the curvature tensor $R$ are given by
$$R_{i\bar{j}k\bar{l}}=
-\frac{\partial^2g_{k\bar{l}}}{\partial z^i\partial\bar{z}^j}+g^{p\bar{q}}\frac{\partial g_{k\bar{q}}}{\partial z^i}\frac{\partial g_{p\bar{l}}}{\partial \bar{z}^j},$$
where $$g_{i\bar{j}}:=
g(\frac{\partial}{\partial z^i},\frac{\partial}{\partial \bar{z^i}}),\quad (g^{j\bar{j}}):=(g_{i\bar{j}})^{-1}.$$

The smooth function $S_g$ on $M$ defined by
$$S_{g}:=2g^{i\bar{j}}g^{k\bar{l}}R_{i\bar{j}k\bar{l}}$$
is the \emph{scalar curvature} of the underlying Riemannian metric of $g$. $S_g$ is called \emph{quasi-positive} if $S_g\geq0$ and is strictly positive at some point.

For $p\in M$ and $X=X^i\frac{\partial}{\partial z^i}\in T_p^{1,0}M$,
the \emph{holomorphic sectional curvature} $H$ of $g$ \big(HSC($g$) for short\big) at the point $p$ and the direction $X$ is defined by
\be\label{hsc}H_p(X):=R(X,\overline{X},X,\overline{X})=R_{i\bar{j}k\bar{l}}\big|_p\cdot X^i\overline{X^j}X^k\overline{X^l}.\nonumber\ee
$H$ is called \emph{nonnegative} if $H_p(X)\geq0$ for any pair $(p,X)$. $H$ is called \emph{quasi-positive} if it is nonnegative, and there exists a point $p\in M$ such that $H_p(X)>0$ for any nonzero $(1,0)$-type vector at $p$.

%As mentioned in front of Theorem \ref{second thm},  both $\text{Ric}(g)$ and $\text{HSC}(g)$ of K\"{a}hler metrics are dominated by bisectional curvature and meanwhile dominate $S_g$. All these dominations can be verified by their definitions except the fact that $\text{HSC}(g)$ dominates $S_g$, whose proof relies on an integral trick. Although this fact is well-known to field experts, this for example is listed in , we still provide a detailed proof below for the reader's convenience as well as for completeness.
The implication needed below is the standard fact that holomorphic sectional curvature determines scalar curvature. For example, it is listed in \cite[p.189]{Zhe} as an exercise. We record the proof for the reader's convenience as well as for completeness.
\begin{lemma}\label{dominate relation}
HSC$(g)$ completely determines $S_g$. In particular, the quasi-positivity of $HSC(g)$ implies that of $S_g$.
\end{lemma}
\begin{proof}
Let $p\in M$ be an arbitrary point. The proof is to average holomorphic sectional curvature of unit vectors at $p$ to obtain $S_g$. Such a trick is usually attributed to Berger.

We choose a \emph{unitary} basis $\{e_1,\ldots,e_n\}$ of $T^{1,0}_pM$. Let
$$\Sigma_p:=\big\{X=\sum_{i=1}^nX^ie_i\in T^{1,0}_pM,~|X|=1\big\}\cong\mathbb{S}^{2n-1},$$ $d\theta(\cdot)$ the spherical measure on $\Sigma_p$, and  $\text{V}(\mathbb{S}^{2n-1})$ the volume with respect to it. Then
\be
\begin{split}\label{average}
&\int_{X\in\Sigma_p}H_p(X)d\theta(X)\\
=&\int_{X\in\Sigma_p}\Big[R(e_i,\overline{e_j},e_k,\overline{e_l})
X^i\overline{X^j}X^k\overline{X^l}\Big]d\theta(X)\\
=&R(e_i,\overline{e_j},e_k,\overline{e_l})
\frac{\delta_{ij}\delta_{kl}+\delta_{il}\delta_{kj}}{n(n+1)}
\text{Vol}(\mathbb{S}^{2n-1})\\
=&\frac{\text{Vol}(\mathbb{S}^{2n-1})}{n(n+1)}S_g(p).
\end{split}
\ee
Here the second equality in (\ref{average}) is due to the classical identity on spherical measure
$$\frac{1}{\text{Vol}(\mathbb{S}^{2n-1})}
\int_{\mathbb{S}^{2n-1}}X^i\overline{X^j}
X^k\overline{X^l}d\theta(X)
=\frac{\delta_{ij}\delta_{kl}+\delta_{il}\delta_{kj}}{n(n+1)},$$
and the third one in (\ref{average}) is due to the symmetry of the curvature tensor $R$ for a K\"{a}hler metric:
$R_{i\bar{j}k\bar{l}}=R_{i\bar{l}k\bar{j}}.$
\end{proof}
\begin{remark}
This averaging trick and related variants appear, for example, in \cite[Lemma 4.1]{LSY}, \cite[Thm 3.1]{Yan}, \cite[Lemma 4.1]{Li21} and \cite[\S3.3]{Li26}.
\end{remark}
Holomorphic sectional curvature and Ricci curvature are widely believed to be of comparable strength. Deeper relationships between them have been obtained recently whenever $\text{HSC}(g)$ is (quasi)-negative (see \cite{WY1,TY,DT,WY2,CLT,LNZ}).

\subsection{Proof of Theorem \ref{second thm}}
Assume, for contradiction, that $M_0$ is homeomorphic to $M$ in Theorem \ref{second thm}. The proof follows the same strategy as Theorem \ref{first thm}.

\begin{assertion}
The manifold $(M,g)$ is projective and simply-connected.
\end{assertion}
\begin{proof}
Quasi-positive holomorphic sectional curvature implies that the Hodge number $h^{0,2}=0$ (\cite{ZZ,Ta}). Since $M$ is K\"{a}hler, the Kodaira embedding theorem then gives projectivity. Moreover, a projective manifold with a quasi-positive holomorphic sectional curvature K\"{a}hler metric is rationally connected (\cite{HW}), and hence simply connected.
\end{proof}
\begin{remark}
For strictly positive holomorphic sectional curvature,  projectivity and rational connectedness were established by Yang (\cite{Yang18}). Related generalizations can be found in \cite{Ni,NZ,Li24, Li26,DN25,Ta} and the references therein.
\end{remark}

\begin{assertion}
We have $\hat{A}(M)=0$.
\end{assertion}
\begin{proof}
As discussed in the proof of Assertion \ref{second assertion}, the manifold $M$ in Theorem \ref{second thm} is spin. Moreover, Lemma \ref{dominate relation} yields that the scalar curvature $S_g$ is quasi-positive. Then Lichnerowicz's classical vanishing theorem (Theorem \ref{Lich-AH}) leads to $\hat{A}(M)=0$.
\end{proof}

\begin{assertion}
Theorem \ref{second thm} is true.
\end{assertion}
\begin{proof}
The remainder of the argument is identical to the proof of Assertion \ref{fourth assertion}.
\end{proof}

\section{Proofs of Theorem \ref{third refined thm} and Corollary \ref{third thm}}\label{proof of 3thm}
\subsection{Proof of Theorem \ref{third refined thm}}\label{proof of 3thm-2}
\begin{proof}
Assume, for contradiction, that a Calabi-Yau manifold $M_0$ is homeomorphic to a smooth manifold $M$ satisfying $\pi_1(M)=0$, $b_{2p+1}(M)=0$ for all $p$, and admitting a nontrivial smooth compact connected Lie group action.

Since $M_0$ is simply connected, the same integral-cohomology argument as in Assertion \ref{second assertion} gives
$$c_1(M_0)=0\in H^{1,1}(M_0;\Z),$$
and hence $w_2(M_0)=0$. The Wu formula argument used above therefore shows that $M$ is spin. The Atiyah-Hirzebruch vanishing theorem (Theorem \ref{Lich-AH}) gives $\hat{A}(M)=0$, and Lemma \ref{hattopoinv} hence yields $\hat{A}(M_0)=0$.

By Lemma \ref{BB lemma},
$$M_0=M_1\times M_2,$$
where $M_1$ is an \emph{odd}-dimensional SCY manifold and $M_2$ is either a point or a product of SCY and hyperK\"{a}hler factors. Write $$\dim M_1=2p_0+1.$$
The K\"{u}nneth formula and Hodge decomposition give
$$b_{2p_0+1}(M)=b_{2p_0+1}(M_0)\geq b_{2p_0+1}(M_1)\geq2h^{0,2p_0+1}(M_1)=2,$$
contrary to the hypothesis on the odd Betti numbers.
\end{proof}

\subsection{Proof of Corollary \ref{third thm}}\label{proof of 3thm-1}
Let $(M,\omega)$ be a symplectic manifold. We refer to \cite[\S3]{Au} for detailed materials discussed below.

A circle action on $M$ is called \emph{symplectic} if it preserves the symplectic form $\omega$. Let $X$ be the generating vector field of a circle action on $(M,\omega)$. A circle acts on $(M,\omega)$ symplectically if and only if the one-form $\omega(X,\cdot)$ is closed. Indeed, the action is symplectic precisely when $\mathcal{L}_{X}\omega=0$, where $\mathcal{L}_{X}$ is the Lie action. Cartan's formula gives
$$d(i_X\omega)=\mathcal{L}_{X}\omega-i_X(d\omega)=0,$$
where $i_X(\cdot)$ is the contraction action. This symplectic circle action is called \emph{Hamiltonian} if the one-form $\omega(X,\cdot)$ is moreover exact, i.e., $\omega(X,\cdot)=\text{d}f$ for some smooth real-valued function $f$ on $M$. This $f$ is called the \emph{moment map} of this Hamiltonian circle action, which is unique up to an additive constant. Note that the fixed-point set of a Hamiltonian circle action is exactly the critical-point set of the moment map $f$, and hence is always nonempty as the points minimizing or maximizing $f$ are always critical.

We use the following standard Morse-theoretic consequences (see \cite{Ki} or \cite[\S3]{Au})
\begin{lemma}\label{Morse function}
Let $(M,\omega)$ be a symplectic manifold admitting a Hamiltonian circle action with isolated fixed points. Then $M$ is simply connected, $H_{\ast}(M;\mathbb{Z})$ is torsion-free and
 $$H_{2p+1}(M;\Z)=0,\quad \text{all $p$}.$$
In particular, all odd-dimensional Betti numbers of $M$ vanish.
\end{lemma}
\begin{proof}
The moment map $f:~M\rightarrow\mathbb{R}$ of a Hamiltonian circle action with isolated fixed points is a \emph{perfect Morse function} whose critical points are exactly the fixed points. The Morse index at each fixed point is twice the number of negative eigenvalues and hence even. The associated CW decomposition therefore has only even-dimensional cells. Since $M$ is connected, it has a single $0$-cell; hence there are no $1$-cells and $\pi_1(M)=0$. The cellular chain complex has zero differentials, which also gives the asserted integral homology statement. In the K\"{a}hler case these facts go back to Frankel (\cite{Fr}).
\end{proof}

Corollary \ref{third thm} now follows immediately from Theorem \ref{third refined thm}.

\section{Proof of Theorem \ref{fifth thm}}\label{proof of 5thm}
\subsection{K\"{a}hler hyperbolic manifolds}\label{kahlerbyperbolic}
Let $(X,g)$ be a Riemannian manifold, \emph{not necessarily compact}, and
$$p:~(\widetilde{X},\widetilde{g})\rightarrow(X,g)$$
the \emph{universal} covering with $\widetilde{g}:=p^{\ast}(g)$. A differential form $\alpha$ on $(X,g)$ is called \emph{$d$-bounded} if $\alpha=d\beta$ and the norm $$\big|\beta\big|_g:=\sup_{x\in X}\big|\beta(x)\big|_{g(x)}<\infty.$$

A form $\alpha$ on $(X,g)$ is called \emph{$\widetilde{d}$-bounded} if $p^{\ast}(\alpha)$ is $d$-bounded on $(\widetilde{X},\widetilde{g})$. The concept of ``$d$-bounded" or ``$\widetilde{d}$-bounded" is of interest only if $X$ or $\widetilde{X}$ is \emph{non}-compact. With this understood, we have the following (\cite[p. 265]{Gr})
\begin{definition}[Gromov]
Let $(M,g)$ be a K\"{a}hler manifold with $\omega$ the associated K\"{a}hler form of $g$. $(M,g)$ is called \emph{K\"{a}hler hyperbolic} if $\omega$ is $\widetilde{d}$-bounded.
\end{definition}
\begin{remark}
Classical examples of K\"{a}hler hyperbolic manifolds include
K\"{a}hler manifolds which are homotopy equivalent to negatively-curved Riemannian manifolds, and
compact quotients of the bounded homogeneous symmetric
domains in $\mathbb{C}^n$ (\cite[p.265]{Gr}, \cite[\S2.2]{CY}). Moreover, submanifolds and products of K\"{a}hler hyperbolic manifolds are still K\"{a}hler hyperbolic. We refer the reader to \cite{CY} and \cite{Li19} for more properties on K\"{a}hler hyperbolic manifolds.
\end{remark}

Gromov's original motivation for introducing K\"{a}hler hyperbolicity was to attack the Singer conjecture on negatively-curved K\"{a}hler manifolds. The detailed treatment of the following materials can be found in \cite{Lu}.

Let $\chi(X)$ and $b^{(2)}_p(X)$ $(0\leq p\leq\dim X)$ be respectively the Euler characteristic and $L^2$-Betti numbers of an even-dimensional smooth manifold $X$. Atiyah's $L^2$-index theorem (\cite{At}) then gives
\be\label{Euler number}\chi(X)=\sum_{p\geq0}(-1)^pb^{(2)}_p(X)\ee

The Singer conjecture asserts that, when $(X,g)$ is a negatively-curved Riemannian manifold of even dimension $2n$, $b^{(2)}_{n}(X)>0$ and $b^{(2)}_p(X)=0$ whenever $p\neq n$. In particular, together with (\ref{Euler number}), this implies another widely open Hopf conjecture: the signed Euler characteristic of a $2n$-dimensional negatively-curved Riemannian manifold $(X,g)$ is positive: $$(-1)^n\chi(X)>0.$$

The major result in \cite{Gr} is that the aforementioned Singer conjecture holds true for
K\"{a}hler hyperbolic manifolds including negatively-curved K\"{a}hler manifolds. In particular, we have
\begin{lemma}[Gromov]\label{nonzero of Euler number}
The signed Euler characteristic of a K\"{a}hler hyperbolic manifold is positive.
\end{lemma}

In contrast to this, Cheeger-Gromov showed the following vanishing result (\cite[Thm 0.2]{CG}).
\begin{lemma}[Cheeger-Gromov]\label{zero of Euler number}
If the fundamental group $\pi_1(X)$ of a smooth manifold $X$ is amenable and infinite, then $b^{(2)}_p(X)=0$ for all $p$. In particular, $\chi(X)=0$.
\end{lemma}

\subsection{Proof of Theorem \ref{fifth thm}}
Assume that the fundamental group of the K\"{a}hler hyperbolic manifold $(M,\omega)$ is isomorphic to that of a Calabi-Yau manifold $M_0$: $\pi_1(M)\cong\pi_1(M_0)$.

\begin{assertion}\label{last assertion}
The group $\pi_1(M)$ is amenable and infinite.
\end{assertion}
\begin{proof}
By Theorem \ref{BB decomposition}, the fundamental group of some finite cover of $M_0$ is free abelian. So $\pi_1(M_0)$ contains a free abelian subgroup of finite index. This means $\pi_1(M_0)$ and hence $\pi_1(M)$ are virtually abelian and hence amenable.

Let $\omega$ be the K\"{a}hler form of $g$ and $$(\widetilde{M},\widetilde{\omega})
\overset{p}{\longrightarrow}(M,\omega)$$
be the (K\"{a}hler) universal covering of $(M,\omega)$ with $\widetilde{\omega}:=p^{\ast}(\omega)$.

If $\pi_1(M)$ were finite, $(\widetilde{M},\widetilde{\omega})$ would be a \emph{compact} K\"{a}hler manifold. However, K\"{a}hler hyperbolicity of $(M,\omega)$ yields that $\widetilde{\omega}=d\beta$ for some one-form $\beta$ on $\widetilde{M}$. Consequently, compactness of $\widetilde{M}$ and Stokes' theorem yield
$$\int_{\widetilde{M}}(\widetilde{\omega})^n=
\int_{\widetilde{M}}(d\beta)^n=\int_{\widetilde{M}}
d\big(\beta\wedge(d\beta)^{n-1}\big)=0,$$
a contradiction. This means $\pi_1(M)$ is infinite.
\end{proof}

Lemma $\ref{nonzero of Euler number}$ gives $\chi(M)\neq0$, whereas Assertion \ref{last assertion} and Lemma \ref{zero of Euler number} give $\chi(M)=0$. This contradiction proves Theorem \ref{fifth thm}.

\section{Sign of scalar curvature on Calabi-Yau manifolds}\label{scalar curvature}
By the Calabi-Yau theorem, every Calabi-Yau manifold admits a Ricci-flat K\"{a}hler metric. The Kazdan-Warner scalar-curvature prescription theorem (\cite{KW}) also implies that every Calabi-Yau manifold of complex dimension $n\geq2$ admits a Riemannian metric with constant \emph{negative} scalar curvature.

We briefly contrast this with the existence or nonexistence of Riemannian metrics with \emph{positive scalar curvature} (\emph{PSC} for short). The answer depends on the dimension.
\begin{proposition}
No Calabi-Yau surface admits a Riemannian metric of positive scalar curvature.
\end{proposition}
\begin{proof}
Let $S_0$ be a Calabi-Yau surface.
By Theorem \ref{BB decomposition}, there exists a finite \'{e}tale cover $$\widetilde{S_0}\longrightarrow S_0$$
such that $\widetilde{S_0}$ is either a complex torus $T$ or a $K3$ surface.

Neither case admits a PSC metric. For a torus this follows from the results of Schoen-Yau (\cite{SY}) and Gromov-Lawson (\cite{GL}). A $K3$ surface is spin and satisfies $\hat{A}(K3)=2\neq0$, so the nonexistence follows from Theorem \ref{Lich-AH}.

If $S_0$ admitted a PSC metric, its pullback to $\widetilde{S_0}$ would again have a PSC metric, a contradiction.
\end{proof}

When $n\geq3$, we have
\begin{example}
Let $E$ be an elliptic curve and $X(5)$ a smooth quintic threefold in $\mathbb{P}^4$. Set

\begin{eqnarray}
M_1:=\underbrace{E\times\cdots\times E}_{\text{$n$ copies}},\quad
M_2:=\left\{ \begin{array}{ll}
X(5),\qquad n=3,\\
~\\
X(5)\times \underbrace{E\times\cdots\times E}_{\text{$n-3$ copies}},\qquad n\geq4,
\end{array} \right.\nonumber
\end{eqnarray}
Both $M_1$ and $M_2$ are $n$-dimensional Calabi-Yau manifolds. The first admits no PSC metric, whereas the second does.

The smooth manifold $M_1$ is a torus and hence enlargeable, so the Gromov-Lawson theorem rules out a PSC metric (cf. \cite[Thm 5.5]{LM}).

The quintic $X(5)$ is a simply connected real six-dimensional manifold and admits a PSC metric (\cite[Cor.4.6]{LM}). Denote such a metric by $g_+$. Let $g_0$ be a flat metric on $E$. Then the product metric
$$g=g_+\oplus\underbrace{g_0\oplus\cdots\oplus g_0}_{\text{$n-3$ copies}}$$
is a PSC metric on $M_2$.
\end{example}

\section{Appendix}\label{appendix}
For completeness, we record the elementary low-dimensional obstruction mentioned in the Introduction.

\begin{proposition}\label{positivesurface}
Let $S_+$ be a Fano surface and $S_0$ a Calabi-Yau surface. Then $S_+$ is not homotopy equivalent to $S_0$.
\end{proposition}
\begin{proof}
By the classification of Fano surfaces, $S_+$ is one of
\be\label{1}\mathbb{P}^2,\quad\mathbb{P}^1\times\mathbb{P}^1,\qquad
\text{Bl}_{p_1,\ldots,p_r}\mathbb{P}^2,~1\leq r\leq 8,\ee
with the points satisfying the usual generic conditions. The list (\ref{1}) implies that
\be\label{2}b_2(S_+)\leq9.\ee

The Calabi-Yau surface $S_0$ is automatically minimal. By the Enriques-Kodaira classification,
$S_0$ is one of the four types of surfaces:
\be\label{3}K3,\quad\text{Enriques},\quad\text{complex torus},\quad\text{bielliptic surface}.\ee

If $S_0$ were homotopy equivalent to $S_+$, then $$\pi_1(S_0)\cong\pi_1(S_+)=0.$$
This immediately excludes Enriques surfaces, complex tori, and bielliptic surfaces from (\ref{3}). Therefore
$S_0$ would have to be a $K3$ surface. But a $K3$ surface has $b_2=22$. This contradicts (\ref{2}).
\end{proof}

\begin{proposition}\label{negativesurface}
Let $S_-$ be a canonically polarized surface and $S_0$ a Calabi-Yau surface. Then $S_-$ is not homotopy equivalent to $S_0$.
\end{proposition}
\begin{proof}
Let $\sigma(S)$ denote the signature of a complex surface $S$. The Hirzebruch signature formula in this case reads
\be\label{eq:surface}
3\sigma(S)=c_1^2(S)-2c_2(S).
\ee

Assume for contradiction that there is a homotopy equivalence $$f:~S_{-}\overset{\simeq}{\longrightarrow}S_0$$

\subsubsection*{Orientation-preserving case}

If \(f\) preserves the orientations, then
$$c_2(S_{-})=c_2(S_0),\quad \sigma(S_{-})=\sigma(S_0).$$ Using (\ref{eq:surface}) we obtain $c_1^2(S_{-})=c_1^2(S_0),$ which is impossible because $c_1^2(S_{-})>0$ whereas $c_1^2(S_0)=0$.

\subsubsection*{Orientation-reversing case}

If \(f\) reverses the orientations, then
\be\label{4}c_2(S_{-})=c_2(S_0),\quad \sigma(S_{-})=-\sigma(S_0).\ee
Using (\ref{eq:surface}) we obtain
\be\label{5}3\sigma(S_{-})=c_1^2(S_{-})-2c_2(S_{-}),\quad
3\sigma(S_{0})=-2c_2(S_{0}).\ee
Putting (\ref{4}) and (\ref{5}) together yields
\be c_1^2(S_{-})=4c_2(S_-).\nonumber\ee

On the other hand, the Miyaoka-Yau inequality gives
$$c_1^2(S_-)\leq 3c_2(S_-).$$
Combining the last two identities yields
\[
4c_2(S_-)\leq 3c_2(S_-).
\]
Moreover \(c_1^2(S_-)>0\), so \(c_2(S_-)>0\).  This is impossible.
\end{proof}

\end{document}